\documentclass[12pt]{amsart}
\usepackage[margin=1.15in]{geometry}
\usepackage{enumerate}
\usepackage{latexsym}

\newcommand{\supp}{Supp}

\numberwithin{equation}{section}
\newtheorem{theorem}{Theorem}[section]
\newtheorem{corollary}[theorem]{Corollary}

\newtheorem{lemma}[theorem]{Lemma}

\newtheorem{remark}[theorem]{Remark}

\newtheorem{definition}[theorem]{Definition}

\usepackage{todonotes}

\begin{document}

\author[V. Asatryan]{Vahram Asatryan}
\address{V. Asatryan, Faculty of Mathematics and Mechanics, Yerevan State University}
\email{vm.asatryan@gmail.com}

\author[E. Babasyan]{Erik Babasyan}
\address{E. Babasyan, Faculty of Mathematics and Mechanics, Yerevan State University}
\email{erik.babasyan.02@gmail.com}

\author[S. Mkrtchyan]{Sevak Mkrtchyan}
\address{S. Mkrtchyan, Department of Mathematics, 
University of Rochester,
500 Joseph C. Wilson Blvd., Rochester, NY 14627}
\email{sevak.mkrtchyan@rochester.edu}

\title[Moment conditions for the asymptotic completeness of iid gap sequences]{Moment conditions for the asymptotic completeness of iid gap sequences}

% \date{}

\begin{abstract}
We study conditions under which integer sequences with independent, identically distributed gaps are asymptotically $k$-complete, meaning that every sufficiently large integer can be represented as the sum of exactly $k$ distinct elements of the sequence, or equivalently whether $k$-fold sumsets with distinct entries from such sequences generate all sufficiently large integers. Prior results established asymptotic completeness under strong conditions on the gap distribution involving the moment generating function. Leveraging renewal theory our main result shows that asymptotic $2$-completeness holds almost surely under the much weaker assumption of a finite second moment. Furthermore, using Schnirelmann densities and Mann's theorem we show weak asymptotic $k$-completeness under only a finite first moment condition, albeit with an upper bound on the first moment.

\noindent
{\it Keywords: Asymptotic completeness; renewal processes; sumsets; additive combinatorics; additive number theory.}

\noindent
2010 {\it Mathematics Subject Classification:} Primary: 60C05. Secondary: 60K05, 11B30, 11P70.

\end{abstract}

\maketitle

\setcounter{tocdepth}{1}
\tableofcontents
\setcounter{tocdepth}{3}

\section{Introduction}

We study asymptotic completeness properties of random sequences with independent, identically distributed gaps. Completeness is defined slightly differently by various authors, so to facilitate the discussion, let us define precisely the various notions of completeness we will study.

\begin{definition}
    Let $\bar{w}=\{w_n\}_{n\geq 1}$ be an increasing sequence of positive integers. We say that $\bar{w}$ is complete, if every positive integer can be written as the sum of distinct elements of $\bar{w}$.
\end{definition}
There are rather simple necessary and sufficient conditions for a sequence to be complete, namely that $a_1=1$ and $a_k\leq 1+\sum_{i=1}^{k-1}a_i$ for $k\geq 2$ (see \cite{Brown61}). If we loosen the requirement that every integer be represented to every large enough integer be represented, that is if we move from completeness to asymptotic completeness, the question becomes much more subtle and no such simple criterion exists.

\begin{definition}
    Let $\bar{w}=\{w_n\}_{n\geq 1}$ be an increasing sequence of positive integers and $k$ a positive integer. We say that $\bar{w}$ is $k$-complete, if every positive integer can be written as the sum of exactly $k$ distinct elements of $\bar{w}$. Replacing $k$ by $\leq k$ will mean at most $k$ distinct elements can be used. Adding the adjective \textit{weakly} means the condition of distinctness is dropped, and adding the descriptor \textit{asymptotic} means instead of every integer, every large enough integer needs to be represented.
\end{definition}
For example, $\bar{w}$ is weakly asymptotically $\leq k$-complete, if every large enough integer can be written as the sum of at most $k$, not necessarily distinct, elements of $\bar{w}$. 

Note that the notion of weak completeness can be phrased in terms of sumsets as well. If $\bar{w}$ is regarded as a set, then, for example, $\bar{w}$ is weakly asymptotically $3$-complete, if the sumset $\bar{w}+\bar{w}+\bar{w}$ contains every large enough integer, where the sumset of two sets $A$ and $B$ is defined to be $A+B=\{a+b:a\in A, b\in B\}$. The concept is also related to the notion of bases (see \cite{TaoVuAdditiveC} for a definition, and many results on completeness). 

Questions of completeness of integer sequences arise in various contexts, and have been studied extensively in additive combinatorics and number theory. The most celebrated such question is the Goldbach conjecture from number theory, stating that every even number other than $2$ can be written as the sum of two primes, or the weak conjecture (now proven by Harald Helfgott \cite{helfgott2014ternarygoldbachconjecturetrue}), stating that every odd integer greater than $5$ can be written as the sum of three primes. 

In this paper we study completeness properties of random sequences. From a number theoretic point of view one such sequence of interest would be Cramer's probabilistic model of primes (\cite{Cramer1936}), where each integer $n$ is declared a ``prime'' independently, with probability $\frac{1}{\log n}$, resulting in a random sequence mimicking the density of primes. However, the fact that the entries of the sequence are chosen independently makes studying completeness rather simple. For an interesting read on the limitations of Cramer-type independent models of primes for predicting properties of the actual primes see \cite{Pintz07}.

In \cite{Cilleruelo2016} the authors studied completeness properties of pseudo $s$-th powers introduced by Erd\"{o}s and R\'{e}nyi \cite{Erdos1960}, where, similarly to Cramer's model, each integer $n$ is independently declared an ``$s$-th power'' with probability $\frac 1s n^{-1+1/s}$, to get the correct density of $s$'th powers of integers. This is related to Waring's problem, which states that for any $s$, there is a positive integer $r(s)$, such that the $s$'th powers are $\leq r(s)$ complete. Because of the built in independence again, in \cite{Cilleruelo2016} they were able to obtain quite robust results regarding completeness and observed some threshold phenomenon occur. For proofs of Waring's problem see \cite{Hilbert1909} for the original proof by Hilbert and \cite{Vinogradov1928} for a simplified one using the circle method.

Instead of selecting the elements of $\bar{w}$ independently, we study the completeness of increasing sequence where instead the gaps between elements are independent random variables. More precisely, we let $\bar{X}=\{X_i\}_{i=1}^\infty$ be an i.i.d. sequence of positive, integer valued random variables, and consider the sequence $\bar{W}=\{W_n\}_{n=1}^\infty$ with $W_n=X_1+\dots+X_n$. Such sequences were studied in \cite{BrMkPak22}, where part of the motivation was the study of the homology of certain simplicial complexes called quota or threshold complexes with i.i.d. weights. Following \cite{BrMkPak22}, we will refer to the elements $W_n$ of $\bar{W}$ as weights, and $X_n$ as gaps. Studying completeness properties of such sequences is more challenging than the ones described above, even though the density here is much higher, because the weights $W_n$ are now highly correlated. 

Note, that if the support $\supp{X}$ of the gap distribution has a greatest common divisor $d>1$, then every integer written as a sum of weights from $\bar{W}$ will be a multiple of $d$, so we will certainly not have any asymptotic completeness. From now on we will make the assumption that $\gcd\{\supp X\}=1$.

Using sieving ideas it was shown in \cite{BrMkPak22} that under a rather strong condition on the gap distribution, involving a bound on the radius of convergence of the moment generating function, a weight sequence generated by i.i.d. gaps is asymptotically $k$-complete with probability $1$ for every $k\geq 2$. 

Our main result is to show that asymptotic $2$-completeness holds almost surely under the much weaker second moment condition, which is a significant improvement over the moment-generating-function condition from \cite{BrMkPak22}.

\begin{theorem}
\label{thm: main2Complete}
Let $\bar{W}=\{W_1,W_2,\dots\}$ be a sequence of positive, integer-valued random variables (``weights'') such that the gaps $\{W_{n}-W_{n-1}\}_{n\in\mathbb{N}}$ are independent, identically distributed, positive, integer-valued random variables. If the distribution of gaps has a finite second moment, then the sequence of weights is asymptotically $2$-complete.
\end{theorem}

Note, that, as one might expect, almost sure $k$-completeness for such sequences implies almost sure $k+1$ completeness as well (see \cite{BrMkPak22} for the simple argument), so we immediately have that under the finite-second-moment condition on the gaps i.i.d. gap sequences are also asymptotically $k$-complete for all $k\geq 2$.

To prove the theorem, we reduce the question to a question regarding intersection points of independent, delayed renewal processes. See for example \cite{KenBer-EJP17} and \cite{Lind86} for results on such processes.

The ideas from renewal theory do not allow us to weaken the finite-second-moment condition. Using the concept of Schnirelmann density and Mann's theorem \cite{Mann42}, we also obtain a weak $k$-completeness result under only a first moment condition, albeit under a rather restrictive upper bound on the first moment.

\begin{theorem}
    \label{thm:main_1st_moment_k_complete}
    Let $\bar{W}=\{W_1,W_2,\dots\}$ be a sequence of positive, integer-valued random variables such that the gaps $\{W_{n}-W_{n-1}\}_{n\in\mathbb{N}}$ are independent, identically distributed, positive, integer-valued random variables. If the distribution of gaps has a finite first moment less than $k$, then the sequence of weights is weakly asymptotically $k$-complete.
\end{theorem}

In the case $k=2$ we strengthen the result from weak $2$-completeness to $2$-completeness.

\begin{corollary}
    \label{cor:main_1st_moment_2_complete}
    Let $\bar{W}=\{W_1,W_2,\dots\}$ be a sequence of positive, integer-valued random variables such that the gaps $\{W_{n}-W_{n-1}\}_{n\in\mathbb{N}}$ are independent, identically distributed, positive integer-valued random variables. If the distribution of gaps has a finite first moment less than $2$, then the sequence of weights is asymptotically $2$-complete.
\end{corollary}

Note, that a much simpler proof of the Waring-Hilbert theorem was obtained by Linnik using the notion of Schnirelmann density \cite{Linnik1943,Jameson2013LINNIKSPO}.

\begin{remark}
A simple application of the Borel-Cantelli lemma shows that if $\mathbb{E}X^{\frac{1}{k}} = \infty$, then $P(X_m \geq m^k \quad i.o.) = 1$, so by the pigeonhole principle, $$P(\{W_n\}_{n\in \mathbb{N}} \text{ is asymptotically $k-$complete}) = 0.$$ In particular, for asymptotic $2$-completeness, at least a finite $1/2$-moment is needed.
\end{remark}

\begin{remark}
    Another approach that can be used to tackle the problem is to utilize the fact that the correlation between weight sequences in far locations decreases as a function of the distance between them. More precisely, given $n$, let $k$ be an integer growing much slower than $n$, let $s_i=\lfloor \frac{n i}{k}\rfloor$ for $0\leq i< k$ and consider the $k$ processes $W^i:=\{W_j:s_i\leq W_j<(s_i+s_{i+1})/2\}$. Conditioned on there not being very large gaps, these processes should be only weakly correlated, since they are far apart. One can try to use this to show asymptotic completeness, but the condition one gets using such an approach is that the gap distribution should have a finite $2+f(\mathbb{E} X)$ moment, where $f$ is a certain increasing function. 
\end{remark}

\subsection*{Outline}
Section \ref{sec:2ndmoment} is devoted to the proof of Theorem \ref{thm: main2Complete}. Section \ref{sec:1stmoment} is devoted to the proofs of Theorem \ref{thm:main_1st_moment_k_complete} and Corollary \ref{cor:main_1st_moment_2_complete}. 
% In the Appendix we give the simple proof of the result that, from density considerations, in order for the sequence $\bar{W}$ to be $k$-complete, the gaps $\bar{X}$ need to have at least a finite $1/k$-moment. 

\subsection*{Acknowledgements}
{
We would like to thank Kenneth Alexander for a helpful discussion regarding renewal processes. SM was partly supported by Simons Foundation Grant No. 422190.  For part of this research SM was visiting Yerevan State University through the support of the Fulbright Scholar program, and would like to thank them for their hospitality.
}

\section{Second moment condition for 2-completeness.}
\label{sec:2ndmoment}

In this section we prove asymptotic $2$-completeness under the finite-second-moment condition. 

Let $M$ be a sufficiently large real number and denote $$L_n:=\left\{\max{(X_1, ..., X_n)} \le \frac{n}{M}\right\}.$$ Note, that a simple Borel-Cantelli argument shows that for any fixed $M>0$, $L_n$ holds almost surely for any large enough $n$. For a positive integer $n$, let $T_n$ be the index of the last weight before reaching $k$, i.e. $T_k = \max\{i \in \mathbb{N} | W_i \le k\}$, and for $1\leq i\leq T_n$ define $$R_i = n - W_{T_n} + \sum\limits_{j=0}^{i - 1}X_{T_n - j}.$$
It is easy to see that if $n$ is not the sum of two weights from $\bar{W}$, then
\begin{equation}
    \label{eq:w_int_r_empty}
    \{W_i\}_{1 \le i \le T_n} \cap \{R_i\}_{1 \le i \le T_n} \cap \left[1, \frac{n}{3}\right] = \emptyset .
\end{equation}

% First, a simple maximum bound.
% \begin{lemma}
% \label{lem:max_gap_evaluation}
% Let $X_1, X_2,  \dots, X_n, \dots$ be $i.i.d.$ random variables with finite expected value. Then $\forall \varepsilon > 0$ with probability $1$, for any large enough $n$ $$\frac{\max\{X_1, X_2, \dots, X_n\}}{n} \leq \varepsilon.$$ 
% \end{lemma}

% \begin{proof}
% We have 
% \begin{equation*}
%     P\left(\frac{\max\{X_1, X_2, \dots, X_n\}}{n} > k \quad i.o.\right) 
%     = P\left(\frac{X_n}{n}\geq \varepsilon \quad i.o. \right),
% \end{equation*}
% however 
% \begin{equation*}
%     \sum_{n} P\left(\frac{X_n}{n}\geq \varepsilon\right)
%     =\sum_n P\left(\frac{X_n}{\varepsilon}\geq n\right)
%     =\sum_n P\left(\frac{X_1}{\varepsilon}\geq n\right),
% \end{equation*}
% which is finite if $X_1/\varepsilon$ has finite expectation, so by the Borel-Cantelli lemma the statement follows.
% \end{proof}

Lastly, for each nonnegative integer $b \in [0, n)$, for any $ i \in [1, T_{n - b}]$, define $$W'_{b, i} = X_{T_{n-b}} + ... + X_{T_{n-b} - i + 1}=W_{T_{n-b}}-W_{T_{n-b}-i},$$
which is the sequence of weights you get if you start accumulating gaps moving backwards from the $T_{n-b}$-th gap.

\begin{lemma}
\label{lem:estimate_empty_intersection}
We have
    \begin{multline*}
        P\left( \{W_i\}_{1 \le i \le T_n} \cap \{R_i\}_{1 \le i \le T_n} \cap \left[1, \frac{n}{3}\right] = \emptyset \right) 
        \le (\mathbb{E} X + 1) \cdot P(L_n^c) \\+ \sum\limits_{b = 0}^{\frac{n}{M}}P(X \ge b + 1)P(\{W_i\}_{1 \le i \le T_{n - b}}\cap\left[1, \frac{n}{3}\right] \cap \{b + W'_{b , i}\}_{1 \le i \le T_{n - b}}, L_n).
    \end{multline*}
\end{lemma}

\begin{proof}
    First, note that 
    \begin{multline}
    \label{eq:cond_on_L_n}
        P\left( \{W_i\}_{1 \le i \le T_n} \cap \{R_i\}_{1 \le i \le T_n} \cap \left[1, \frac{n}{3}\right] = \emptyset \right) 
        % \\=P\left( \{W_i\}_{1 \le i \le T_n} \cap \{n - W_{T_n} + X_{T_n} + ... + X_{T_n - i + 1}\}_{1 \le i \le T_n} \cap \left[1, \frac{n}{3}\right] = \emptyset \right) 
        \le P(L_n^c) \\+ 
        P\left( \{W_i\}_{1 \le i \le T_n} \cap \{n - W_{T_n} + X_{T_n} + ... + X_{T_n - i + 1}\}_{1 \le i \le T_n} \cap \left[1, \frac{n}{3}\right] = \emptyset, L_n \right).
    \end{multline}
    Denoting $b = n - W_{T_n}$ and noting that under $L_n$, $b < \frac{n}{M}$, the last term can be written as
    \begin{multline*}
        % P\left( \{W_i\}_{1 \le i \le T_n} \cap \{n - W_{T_n} + X_{T_n} + ... + X_{T_n - i + 1}\}_{1 \le i \le T_n} \cap \left[1, \frac{n}{3}\right] = \emptyset, L_n \right) &= \\
        \sum\limits_{b = 0}^{\frac{n}{M}}P\left( \{W_i\}_{1 \le i \le T_n} \cap \{b + X_{T_n} + ... + X_{T_n - i + 1}\}_{1 \le i \le T_n} \cap \left[1, \frac{n}{3}\right] = \emptyset, 
        \right.\\\left. L_n, n - W_{T_n} = b, X_{T_n + 1} \ge b + 1 \right),
    \end{multline*}
    which, by summing over the possible values $u=T_n$, noting that under $L_n$, $T_n \ge M$, and then dropping the intersection with $L_n$, is bounded above by
    \begin{multline*}
        % \sum\limits_{b = 0}^{\frac{n}{M}}P\left( \{W_i\}_{1 \le i \le T_n} \cap \{b + X_{T_n} + ... + X_{T_n - i + 1}\}_{1 \le i \le T_n} \cap \left[1, \frac{n}{3}\right] = \emptyset, L_n, n - W_{T_n} = b, X_{T_n + 1} \ge b + 1 \right) &\le \\
        % \sum\limits_{b = 0}^{\frac{n}{M}}\sum\limits_{u = M}^{n}P\left( \{W_i\}_{1 \le i \le u} \cap \{b + X_{u} + ... + X_{u - i + 1}\}_{1 \le i \le u} \cap \left[1, \frac{n}{3}\right] = \emptyset, L_n, n - W_{u} = b, X_{u + 1} \ge b + 1 \right) &\le \\
        \sum\limits_{b = 0}^{\frac{n}{M}}\sum\limits_{u = M}^{n}P\left( \{W_i\}_{1 \le i \le u} \cap \{b + X_{u} + ... + X_{u - i + 1}\}_{1 \le i \le u} \cap \left[1, \frac{n}{3}\right] = \emptyset, 
        \right.\\\left. n - W_{u} = b, X_{u + 1} \ge b + 1 \right).
    \end{multline*}
    Noting that $X_{u+1}$ is independent of $X_1, ..., X_u$, the last expression can be written as  
    \begin{multline*}
        \sum\limits_{b = 0}^{\frac{n}{M}}P(X \ge b + 1) \\\times\sum\limits_{u = M}^{n}P\left( \{W_i\}_{1 \le i \le u} \cap \{b + X_{u} + ... + X_{u - i + 1}\}_{1 \le i \le u} \cap \left[1, \frac{n}{3}\right] = \emptyset, n - W_{u} = b\right).
    \end{multline*}
    The inner sum can now be estimated from above by
    \begin{align*}
        \sum\limits_{u = M}^{n}& P\left( \{W_i\}_{1 \le i \le u} \cap \{b + X_{u} + ... + X_{u - i + 1}\}_{1 \le i \le u} \cap \left[1, \frac{n}{3}\right] = \emptyset, T_{n - b} = u \right) 
        \\ &=
        \sum\limits_{u = M}^{n}P\left( \{W_i\}_{1 \le i \le u} \cap \{b + W'_{b , i}\}_{1 \le i \le T_{n - b}} \cap \left[1, \frac{n}{3}\right] = \emptyset, T_{n - b} = u \right) 
        \\ &\leq
        P\left( \{W_i\}_{1 \le i \le T_{n - b}} \cap \{b + W'_{b , i}\}_{1 \le i \le T_{n - b}} \cap \left[1, \frac{n}{3}\right] = \emptyset \right) 
        \\ &\le 
        P\left( \{W_i\}_{1 \le i \le T_{n - b}} \cap \{b + W'_{b , i}\}_{1 \le i \le T_{n - b}} \cap \left[1, \frac{n}{3}\right] = \emptyset, L_n \right) +  P(L_n^c).
    \end{align*}
    Multiplying by $P(X\geq b+1)$, summing over $b$ between $0$ and $\frac nM$ and combining with  \eqref{eq:cond_on_L_n} gives the desired result.
\end{proof}

Let $X_{[a,b]}$ denote the section $X_a,\dots,X_b$ of the sequence $\{X_i\}_i$. For a given vector $\vec{v}=(x_1,x_2,\dots,x_m)$,
let us define $\vec{v}' = (x_1, x_2, ..., \dots, x_{m-1})$, 
$S(\vec{v}) = \sum_{i = 1}^{m}x_i$, and $l(\vec{v}) = m$.

Let $Y_1, Y_2, ...$ be an independent copy of the i.i.d. gap sequence $X_1, X_2, ...$ and denote the sequence of weights resulting from this sequence of gaps by $U_1, U_2, ...$, i.e. $U_n=Y_1+\dots+Y_n$.

\begin{lemma}
\label{lem:pseudo_independence}
Let $K$ be some set of vector pairs $(\vec{v}_1, \vec{v}_2)$ such that $\frac{n}{3} \le S(\vec{v}_1) \le \frac{4n}{9}, S(\vec{v}_1' < \frac{n}{3}), \frac{n}{3} \le S(\vec{v}_2) \le \frac{4n}{9}, S(\vec{v}_2' < \frac{n}{3}) $.
Then, for any fixed nonnegative integer $b \le \frac{n}{M}$, 
\begin{multline*}
P\left(\bigcup\limits_{(\vec{v}_1, \vec{v}_2) \in K}\{ X_{[1, l(\vec{v}_1)]} = \vec{v}_1, X_{[T_{n - b} - l(\vec{v}_2) + 1, T_{n - b}]} = \vec{v}_2\}, L_n\right) 
\\ \le \sum\limits_{(\vec{v}_1, \vec{v}_2) \in K}P(X_{[1, l(\vec{v}_1)]} = \vec{v}_1, Y_{[1, l(\vec{v}_2)]} = \vec{v}_2).
\end{multline*}

\end{lemma}

\begin{proof}
Note that
    \begin{multline*}
        P\left(\bigcup\limits_{(\vec{v}_1, \vec{v}_2) \in K}\{ X_{[1, l(\vec{v}_1)]} = \vec{v}_1, X_{[T_{n - b} - l(\vec{v}_2) + 1, T_{n - b}]} = \vec{v}_2\}, L_n\right) \\ = \sum\limits_{(\vec{v}_1, \vec{v}_2) \in K}P(X_{[1, l(\vec{v}_1)]} = \vec{v}_1, X_{[T_{n - b} - l(\vec{v}_2) + 1, T_{n - b}]} = \vec{v}_2, L_n).
    \end{multline*}
    Let us consider any pair $(\vec{v}_1, \vec{v}_2) \in K$. Note that under $L_n$ and due to $M$ being sufficiently big, 
    $$W_{T_{n - b}} \ge n - b - \frac{n}{M} \ge n - 2\frac{n}{M} = n\left(1 - \frac{2}{M} \right) > \frac{8n}{9}.$$
    Thus, when $X_{[1, l(\vec{v}_1)]} = \vec{v}_1, X_{[T_{n - b} - l(\vec{v}_2) + 1, T_{n - b}]} = \vec{v}_2$, $T_{n - b} > l(\vec{v}_1) + l(\vec{v}_2)$ since otherwise $$W_{T_{n-b}} = \sum\limits_{i = 1}^{T_{n - b}}X_i \le \sum\limits_{i = 1}^{l(\vec{v}_1)}X_i + \sum\limits_{i = T_{n-b} - l(\vec{v}_2) + 1}^{T_{n - b}}X_i = S(\vec{v}_1) + S(\vec{v_2}) \le \frac{8n}{9},$$ which is a contradiction . Thus,
    \begin{align*}
        P(X_{[1, l(\vec{v}_1)]} = \vec{v}_1, & X_{[T_{n - b} - l(\vec{v}_2) + 1, T_{n - b}]} = \vec{v}_2, L_n) \\&=
        \sum\limits_{u = l(\vec{v}_1)+ l(\vec{v}_2) + 1}^{n}P(X_{[1, l(\vec{v}_1)]} = \vec{v}_1, X_{[u - l(\vec{v}_2) + 1, u]} = \vec{v}_2, L_n, T_{n - b} = u) \\&=
        \sum\limits_{u = l(\vec{v}_1) + l(\vec{v}_2) + 1}^{n}P(X_{[1, l(\vec{v}_1)]} = \vec{v}_1, X_{[l(\vec{v}_1) + 1, l(\vec{v}_1) + l(\vec{v}_2))]} = \vec{v}_2, L_n, T_{n - b} = u) \\&=
        P(X_{[1, l(\vec{v}_1)]} = \vec{v}_1, X_{[l(\vec{v}_1) + 1, l(\vec{v}_1) + l(\vec{v}_2))]} = \vec{v}_2, L_n) 
        \\ &\le P(X_{[1, l(\vec{v}_1)]} = \vec{v}_1, X_{[l(\vec{v}_1) + 1, l(\vec{v}_1) + l(\vec{v}_2))]} = \vec{v}_2) 
        \\ &=  P(X_{[1, l(\vec{v}_1)]} = \vec{v}_1, Y_{[1, l(\vec{v}_2))]} = \vec{v}_2).
    \end{align*}
    Summing over all the pairs in $K$ finishes the proof of the lemma.
\end{proof}

\begin{corollary}
\label{cor: independence_usage}
For each $b \in \left[0, \frac{n}{M}\right]$ we have
\begin{multline*}
    P\left( \{W_i\}_{1 \le i \le T_{n - b}} \cap \{b + W'_{b , i}\}_{1 \le i \le T_{n - b}} \cap \left[1, \frac{n}{3}\right] = \emptyset, L_n \right) \\
    \le P\left( \{W_i\} \cap \{b + U_i\} \cap \left[1, \frac{n}{3}\right] = \emptyset\right). 
\end{multline*}
\end{corollary}
\begin{proof}
This follows from Lemma \ref{lem:pseudo_independence}, by taking $ K $ to be the set of all pairs of vectors $ (\vec{v}_1, \vec{v}_2 ) $ such 
that no coordinate of $\vec{v}_1$ is greater than some coordinate of $\vec{v}_2$ by exactly $b$.
\end{proof}

\begin{proof}[Proof of Theorem \ref{thm: main2Complete}]
    For each $n \in \mathbb{N}$, let $A_n$ denote the event that $n$ is not the sum of $2$ distinct weights in our process.
    For a nonnegative integer $b$, let $$K_{0, b} = \min\{i | W_i = b + U_i\}.$$ 
    Using \eqref{eq:w_int_r_empty}, Lemma \ref{lem:estimate_empty_intersection} and Corollary \ref{cor: independence_usage}, we have
    
    \begin{multline*}
        P(A_n)
        \le P\left( \{W_i\}_{1 \le i \le T_n} \cap \{R_i\}_{1 \le i \le T_n} \cap \left[1, \frac{n}{3}\right] = \emptyset \right) \\
        \le \sum\limits_{b = 0}^{\frac{n}{M}}\left(P(X \ge b + 1) P\left( \{W_i\} \cap \{b + U_i\} \cap \left[1, \frac{n}{3}\right] = \emptyset\right) \right) + (\mathbb{E}X + 1) \cdot P(L_n^c)
    \end{multline*}
    Summing over $n$, exchanging the order of summation, and extending the range of $b$ to all non-negative integers, we get
    \begin{align*}
        \sum\limits_{n = 1}^{\infty}P(A_n) &- (\mathbb{E}X + 1) \cdot \sum\limits_{n = 1}^{\infty} P(L_n^c) 
        \\ &\le
        \sum\limits_{b = 0}^{\infty}P(X \ge b + 1)\sum\limits_{n = 1}^{\infty}\left( P\left( \{W_i\} \cap \{b + U_i\} \cap \left[1, \frac{n}{3}\right] = \emptyset\right) \right)
        \\&\le
        \sum\limits_{b = 0}^{\infty}P(X \ge b + 1)\sum\limits_{n = 1}^{\infty}P\left(K_{0,b} > \frac{n}{3} \right).
    \end{align*}
    The inequality $K_{0,b} > \frac{n}{3}$ can be interpreted as the event that two independent renewal processes generated by the same distribution, with one having a delayed start at $b$ do not interstect until after time $n/3$. It follows from the proof of Proposition 2 in \cite{Lind86} that for some constant $C$ and every nonnegative integer $b$, $$\mathbb{E}K_{0, b} \le C \cdot b.$$
    Using this estimate we obtain
    \begin{align*}
        \sum\limits_{b = 0}^{\infty}P(X \ge b + 1)\sum\limits_{n = 1}^{\infty}P\left(K_{0,b} > \frac{n}{3} \right) &\le
        \sum\limits_{b = 0}^{\infty}P(X \ge b + 1)3\mathbb{E}K_{0, b}
         \\&\le
        3C \cdot \sum\limits_{b = 0}^{\infty}b \cdot P(X \ge b + 1)\leq C_1\mathbb{E}X^2,
    \end{align*}
    for some constant $C_1$.
    On the other hand, we have
    \begin{align*}
    \label{Ln^c_sum_convergence}
        \sum_{n=1}^{\infty} P(L_n^c) 
        &= \sum_{n=1}^{\infty} P\left(\max \{X_1, X_2, \dots, X_n\} > \frac{n}{M} \right) \\
        &= \sum_{n=1}^{\infty}{P\left(\bigcup_{k=1}^{n}\left( X_k > \frac{n}{M}\right) \right)} \leq \sum_{n=1}^{\infty} \sum_{k=1}^{n} P \left(X_k > \frac{n}{M} \right) \\
        &= \sum_{n=1}^{\infty} n P(MX > n)\leq \mathbb{E}(MX)^2.
    \end{align*}
    Combining these two estimates, we obtain that $\sum_{n=1}^\infty P(A_n)$ is bounded above by a constant multiple of the second moment of the gaps, so if the second moment is finite, 
    the statement of the theorem follows from the Borel-Cantelli lemma.
\end{proof}

\section{First moment condition for weak k-completeness}
\label{sec:1stmoment}

We start by proving a few general facts about $k-$completeness and its connection to weak and $\leq k$ completeness. We combine these results with Mann's theorem to obtain weak $k$-completeness under a first moment condition, and in the case of $k=2$ extend this from weak $2$-completeness to $2$-completeness.

Given a random sequence of gaps $\bar{X}=\{X_i\}_{i \geq 1}$ let $C^{=k}_n(\bar{X})$ denote the event that the subsequence of gaps $\{X_i\}_{i \geq n}$ generates an asymptotically $k$-complete sequence sequence of weights $\bar{W}$. Similarly, denote by $C^{=k,w}_n(\bar{X})$ the event that the subsequence of gaps $\{X_i\}_{i \geq n}$ generates a weakly asymptotically $k$-complete sequence and by $C^{\leq k, w}_n(\bar{X})$ the event that the subsequence of gaps $\{X_i\}_{i \geq n}$ generates a weakly asymptotically at most $k$-complete sequence. For brevity, we will omit the dependence on $\bar{X}$ from the notation.
 
\begin{lemma}
\label{lem:tail_event}
    Generating an asymptotically $k$-complete sequence is a tail event. The same is true for asymptotically weakly $k$-completeness.
    % $P(\{W_n\}_{n\in \mathbb{N}} \text{ is asymptotically $k-$complete}) = 0 \text{ or } 1$
\end{lemma}
    
\begin{proof}
    The proofs are identical, so we only present the argument for the first case.
    Since the sequence $\bar{X}$ is an i.i.d. sequence, we have $P(C^{=k}_1) = P(C^{=k}_{m+1})$. If we fix the gaps $ X_1 = x_1, X_2 = x_2, \dots, X_m = x_m$, then all the weights $W_n$ for $n\geq m$ are the same as the weights generated by the sequence $\{X_i\}_{i \geq m+1}$ but shifted by the constant $\sum_{i=1}^m x_i$, so if the latter is asymptotically $k$-complete, so is the original sequence $\{W_n\}_{n\geq 1}$, implying 
    $$P(C^{=k}_1 \mid X_1 = x_1, X_2 = x_2, \dots, X_m = x_m) \geq P(C^{=k}_{m+1}).$$ Conditioning on the first $m$ gaps we thus obtain
    \begin{align*}
        P(C^{=k}_1) &= \sum_{ \{ x_1,\dots, x_m \} \subseteq \supp X} 
        P(C^{=k}_1 \mid X_1 = x_1, \dots, X_m = x_m) P(X_1 = x_1, \dots, X_m = x_m) \\
        & \geq P(C^{=k}_{m+1}) = P(C^{=k}_1),
    \end{align*}
    where $\supp X$ refers to the support of the gap distribution.
    Equality holds if  
    $$P(C^{=k}_1 \mid X_1 = x_1, X_2 = x_2, \dots, X_m = x_m) = P(C^{=k}_1),$$  
    which implies $ C^{=k}_1 $ is a tail event.
\end{proof}

Note that the same argument does not apply for  $\leq k$-completeness as it's not clear why the inequalities $$P(C^{\leq k}_1 \mid X_1 = x_1, X_2 = x_2, \dots, X_m = x_m) \geq P(C^{\leq k}_{m+1})$$ $$P(C^{\leq k,w}_1 \mid X_1 = x_1, X_2 = x_2, \dots, X_m = x_m) \geq P(C^{\leq k,w}_{m+1})$$ hold. However we shall prove next that $C^{=k,w}_1$ and $C^{\leq k,w}_1$ have the same probability.

\begin{lemma}
\label{lem:nonprob_leq_equiv_equal}
    For a sequence $s$ if $\gcd(\{s_i\}_{i \geq 0}) = \gcd(\{s_i - s_0\}_{i \geq 1})$ then: 
    $\{s_i\}_{i \geq 0}$ is weakly asymptotically $k$ complete if and only if $\{s_i - s_0\}_{i \geq 1}$ is weakly asymptotically $\leq k$ complete.
\end{lemma}

\begin{proof}
    If $\{s_i\}_{i \geq 0}$ is weakly asymptotically $k$-complete, then for any $n$ large enough, there are indices $i_1,\dots,i_k$ such that $n=s_{i_1}+\dots+s_{i_k}$. Then $n-ks_0=(s_{i_1}-s_0)+\dots+(s_{i_k}-s_0)$. If any $i_j$'s are $0$, then the corresponding summands will be zero, so we get that $\{s_i - s_0\}_{i \geq 1}$ is weakly asymptotically $\leq k$-complete. On the other hand, if $\{s_i - s_0\}_{i \geq 1}$ is weakly asymptotically $\leq k$-complete, then any big enough number can be expressed as sum of $t \leq k$ elements from $\{s_i - s_0\}_{i \geq 1}$ and thus any big enough number can be expressed as the sum of $k-t$ $s_0$'s and $t$ elements from $\{s_i\}_{i \geq 1}$, meaning $\{s_i\}_{i \geq 0}$ is weakly asymptotically $k$-complete.
\end{proof}

\begin{lemma}
\label{lem:C=D}
    Given any two integers $n,m\in\mathbb{N}$ we have
    $$P(C^{=k,w}_m) = P(C^{\leq k,w}_n).$$
\end{lemma}

\begin{proof}
    Since the sequences $\{X_i\}_{i\geq n}$ and $\{X_i\}_{i\geq m}$ have the same distribution, and being asymptotically $k$-complete is stronger than being asymptotically $\leq k$-complete, we have $$P(C^{=k,w}_m) = P(C^{=k,w}_n) \leq P(C^{\leq k,w}_n).$$
    
    For $n \geq 2$, if $C^{\leq k,w}_n$ is true then 
    $\{\sum_{i=1}^{t}{X_i} - \sum_{i=1}^{n-1}{X_i}\}_{t \geq n}$ is weakly asymptotically $\leq k$-complete, which by Lemma \ref{lem:nonprob_leq_equiv_equal} implies $\{\sum_{i=1}^{n-1}{X_i}, \{\sum_{i=1}^{t}{X_i} \}_{t \geq n} \}$ is weakly asymptotically $k$-complete, so $P(C^{\leq k,w}_n) \leq P(C^{=k,w}_{n-1}) = P(C^{=k,w}_m)$
\end{proof}

From now on we will drop the subscript from the notation $C^{\leq k,w}$.

For a sequence $s$ denote by $\sigma(s)$ and $d(s)$ respectively its Schnirelmann density and asymptotic density. Namely,
$$\sigma(s) = \inf_{n}{\frac{A(n)}{n}},$$
and
$$d(s) = \lim_{n \rightarrow \infty}{\frac{A(n)}{n}} \quad \text{if it exists},$$
where $A(n) = |\{i: s_i \leq n\}|$.

Recall Mann's theorem, relating sumsets to Schnirelmann densities:
\begin{theorem}[Mann's Theorem \cite{Mann42}]
\label{thm:Mann}
    For subsets $A$ and $B$ of $\mathbb{N}$, if the sumset $A + B \neq \mathbb{N}$, then $$\sigma(A+B) \geq \sigma(A) + \sigma(B) \quad \text{holds.}$$
\end{theorem}

We now use Mann's theorem together with the results above to establish Theorem \ref{thm:main_1st_moment_k_complete}.

\begin{proof}[Proof of Theorem \ref{thm:main_1st_moment_k_complete}]
    By Lemma \ref{lem:C=D} it is enough to prove $P(C^{\leq k,w}) = 1$. By Lemma \ref{lem:tail_event} $C^{k,w}$ is a tail event, and by Lemma \ref{lem:C=D} we have $P(C^{\leq k,w})=P(C^{k,w})$, so it is enough to prove that $P(C^{\leq k,w}) > 0$. From $\mathbb{E}X < k$ if follows that $d(\{W_n\}_{n \in \mathbf{N}}) > \frac{1}{k}$.
    Introduce the random variable $N$ by
    $$N = \min\left\{n:\forall m \geq n, \quad \frac{A(m)}{m} > \frac{1}{k}\right\}.$$
    Since $\sum_{n=1}^{\infty}{P(N=n)} = 1$, it follows that $\exists n_0$ such that $P(N=n_0)>0$.

    We now distinguish two cases.

    \textit{Case 1: $1\in\supp{X}$.}

    We have $$P \left( N=n_0 \text{ for the sequence generated by }\{X_i\}_{i> n_0}\right) > 0,$$
    and since $1\in\supp{X}$, we have
    $$P(X_1 = X_2 = \dots =X_{n_0} = 1) > 0.$$
    Since these two events deal with distinct parts of the i.i.d. sequence $\{X_i\}_{i> n_0}$, with positive probability both events occur simultaneously. Now, let us evaluate the Schnirelmann density $\sigma(\{W_i\}_{i \geq 1})$ for such sequences of gaps. We have    
    $$ \frac{A(n)}{n} \geq
        \begin{cases}
            1 & \text{if } n \leq n_0 \\
            \frac{1}{2} & \text{if } n_0 < n \leq 2n_0 \\
            \frac{1}{k} & \text{if } n > 2n_0
        \end{cases},
    $$
    where the first two inequalities follow from the choice of the first $n_0$ gaps and the last one from the condition $N=n_0$.

    Thus we get $P \left( \sigma(W) \geq \frac{1}{k} \right)>0$ and Mann's theorem implies asymptotic $\leq k$-comp-leteness with the same positive probability.

    \textit{Case 2: $1\notin\supp{X}$.}

    If $\supp{X}$ is finite, than $X$ has all moments, so asymptotic completeness follows from Theorem \ref{thm: main2Complete}, so we can assume $\supp{X}$ is infinite. Define $$x'= \min\{x \in \supp\{X\} : x-1\in Mon\{\supp\{X\}\}\},$$ where $Mon\{A\}$ is the monoid generated by $A$. Note, that $x'$ exists since $\gcd\{\supp{X}\}=1$. 
    
    Define $x_0 = \min\{x\in \supp\{X\}\}$ and note that $\mathbb{E}X<k$ implies $x_0 \leq k-1$.

    Similarly to the previous case, we have 
    $$P(X_1 = x', \quad X_2 = X_3=\dots=X_{n_0+2} = x_0) > 0$$
    and
    $$P \left( N=n_0 \text{ for the sequence generated by $\{X_i\}_{i> n_0+2}$} \right) > 0,$$
    and, as before, these two events deal with distinct parts of the i.i.d. sequence $\{X_i\}_{i> n_0}$, so with positive probability both events occur simultaneously. Let us evaluate the Schnirelmann density $\sigma(\{W_i - (x' - 1)\} )$ on this event. We have
    $$ \frac{A(n)}{n} \geq
        \begin{cases}
            1 & \text{if } n = 1 \\
            \frac{1}{x_0} > \frac{1}{k} & \text{if } 1 < n \leq (n_0+1)x_0 \\
            \frac{1}{x_0+1} \geq \frac{1}{k} & \text{if } (n_0+1)x_0 < n \leq (n_0+1)(x_0+1) \\
            \frac{1}{k} & \text{if } n > (n_0+1)(x_0+1)
        \end{cases}.
    $$
    Here the first three inequalities hold because of the choice of $X_1, \dots X_{n_0+2}$, and the last one because $N=n_0$ for the sequence generated by $\{X_i\}_{i> n_0+2}$.

    Thus, using Mann's theorem, we get that 
    $\{W_i - (x' - 1)\}_{i \geq1}$ is weakly asymptotically $\leq k$-complete. Now, Lemma \ref{lem:nonprob_leq_equiv_equal} implies that $\{(x'-1), \{W_i\}_{i\geq 1}\}$ is weakly asymptotically $k$-complete. By the choice of $x'$ with nonzero probability there are $X_{-m}, \dots X_{0}$ such that their sum equals $x' - 1$ and hence the sequence generated by $\{X_i\}_{i \geq -m}$ is asymptotically  $k$ complete with positive probability.
\end{proof}

We now fix $k = 2$ and prove that we can go from weak $2$-completeness to $2$-completeness.

% \begin{lemma}
% \label{lem:2W_in_W+W}
%     If $\mathbb{E}X < \infty$ then $P(C^{=2,w}_1) = P(C^{=2}_1)$.
% \end{lemma}

\begin{proof}[Proof of Corollary \ref{cor:main_1st_moment_2_complete}]
    Using Theorem \ref{thm:main_1st_moment_k_complete} it is enough to show that if $\mathbb{E}X<\infty$, then $P(C^{=2,w}_1) = P(C^{=2}_1)$. Note that $$P(C^{=2,w}_1, \neg C^{=2}_1) \leq P(\forall i \neq j, \quad 2W_n \neq W_i + W_j \quad n-i.o.),$$ so it is enough to show that for large enough $N$, $\forall n>N$ we have $2W_n \in W + W$ with  probability $1$. To evaluate $P(\forall i \neq j, \quad 2W_n \neq W_i + W_j \quad n-i.o.)$,
    let us compute
    $$\sum_{n=1}^\infty{P(\forall i \neq j, \quad 2W_n \neq W_i + W_j)} = \sum_{n=1}^\infty{P(\forall i < j : W_n - W_i \neq W_j - W_n)}.$$
    Here $\{W_n - W_i\}_{i<n}$ and $\{W_j - W_n\}_{j>n}$ are two independent random walks, while $P(\forall i < j : W_n - W_i \neq W_j - W_n)$ calculates the probability that $I\geq n$, where $I$ is a random variable which is the index of the intersection point of those two random walks. From Blackwell's theorem (\cite{FellerVol1}) $\mathbb{E}I<\infty$, so
    $$\sum_{n=1}^\infty{P(\forall i < j : W_n - W_i \neq W_j - W_n)} = \sum_{n=1}^\infty{P(I \geq n)} = \mathbb{E}I < \infty.$$ Thus $P(\forall i \neq j \quad  2W_n \neq W_i + W_j \quad n-i.o.) = 0$ implying $P(C^{=2,w}_1) = P(C^{=2}_1)$, as needed.
\end{proof}

% \section{Appendix}
% While the main result provides an estimate for a gap moment to reach asymptotic $2-$completeness, here we establish an estimate for which $\{W_n\}_{n\in N}$ is not asymptotically $k-$complete with probability $1$.

% \begin{lemma}
% \label{lem:gap_moment_lower_bound}
%     If $\mathbb{E}X^{\frac{1}{k}} = \infty$ then $P(\{W_n\}_{n\in N} \text{ is asymptotically $k-$complete}) = 0$ 
% \end{lemma}

% \begin{proof}
%     \begin{align*}
%         \mathbb{E}X^{\frac{1}{k}} &= \infty \implies \sum_{n=1}^{\infty} P(X = n) n^{\frac{1}{k}} = \infty \\
%         &\implies \sum_{n=1}^{\infty} P(X = n) \lfloor n^{\frac{1}{k}} \rfloor = \infty \\
%         &\implies \sum_{m=1}^{\infty} \sum_{n=m^k}^{\infty} P(X = n) = \infty \\
%         &\implies \sum_{m=1}^{\infty} P(X_m \geq m^k) = \infty.
%     \end{align*}
    
%     By Borel - Cantelli's lemma,  $P(X_m \geq m^k \quad i.o.) = 1$.
    
%     By the Pigeonhole Principle, $ P(\{W_n\}_{n\in \mathbb{N}} \text{ is asymptotically $k-$complete}) = 0.$
% \end{proof}

% As an immediate result: 
% $$\mathbb{E}X^{\frac{1}{2}} = \infty \implies P(\{W_n\}_{n\in N} \text{ is asymptotically $2-$complete}) = 0$$

\bibliographystyle{alpha}
\bibliography{randomgaps}

\end{document}